\documentclass[11pt,reqno]{amsart}
\usepackage{a4wide,fullpage}
\usepackage{amssymb}
\usepackage{amsthm}
\usepackage{amsmath}
\usepackage{amscd}
\usepackage{verbatim,url}
\usepackage{xcolor,cancel}
\usepackage[mathscr]{eucal}
\usepackage{mathrsfs}

\usepackage[OT2,T1]{fontenc}
\DeclareSymbolFont{cyrletters}{OT2}{wncyr}{m}{n}
\DeclareMathSymbol{\Sha}{\mathalpha}{cyrletters}{"58}

\usepackage{enumitem}
\usepackage{bm}
\usepackage{cleveref}
\allowdisplaybreaks

\setlist[itemize]{leftmargin=*}
\setlist[enumerate]{leftmargin=*,label=\rm{(\arabic*)}}

\newtheorem{theorem}{Theorem}
\newtheorem{lemma}[theorem]{Lemma}

\newtheorem{proposition}[theorem]{Proposition}
\newtheorem*{remark}{Remark}

\theoremstyle{remark}

\numberwithin{theorem}{section} \numberwithin{equation}{section}
\numberwithin{figure}{section}

\newcommand{\R}{\mathbb R}
\newcommand{\N}{\mathbb N}
\newcommand{\C}{\mathbb C}

\newcommand{\Z}{\mathbb{Z}}

\newcommand{\Ec}{\mathcal{E}}

\newcommand{\Oc}{\mathcal{O}}

\renewcommand{\a}{\alpha}
\renewcommand{\b}{\beta}
\renewcommand{\d}{\delta}
\newcommand{\e}{\varepsilon}
\newcommand{\g}{\gamma}
\renewcommand{\l}{\lambda}
\newcommand{\del}{\partial}
\newcommand{\p}{\varrho}

\renewcommand{\t}{\tau}
\renewcommand{\th}{\theta}
\newcommand{\vth}{\vartheta}

\newcommand{\z}{\zeta}
\newcommand{\De}{\Delta}
\newcommand{\Th}{\Theta}

\newcommand{\sm}{\setminus}

\newcommand{\sgn}{\operatorname{sgn}}

\newcommand{\erfc}{\operatorname{erfc}}

\newcommand{\wt}[1]{\widetilde{#1}}

\newcommand{\ol}[1]{\overline{#1}}

\newcommand{\flo}[1]{\lfloor #1\rfloor}

\newcommand{\pmat}[1]{\left(\begin{smallmatrix}#1\end{smallmatrix}\right)}

\newcommand{\leg}[2]{\left(\frac{#1}{#2}\right)}
\newcommand{\tleg}[2]{\left(\tfrac{#1}{#2}\right)}
\newcommand{\pd}[2]{\frac{\del^{#2}}{\del #1^{#2}}}
\renewcommand{\pmod}[1]{\ \left( \mathrm{mod} \, #1 \right)}
\newcommand{\Pmod}[1]{\ ( \mathrm{mod} \, #1 )}
\newcommand{\ou}{\mathrm{ou}}

\makeatletter
\@namedef{subjclassname@2020}{%
	\textup{2020} Mathematics Subject Classification}
\makeatother

\title{Odd unimodal sequences}

\author{Kathrin Bringmann}
\address{University of Cologne, Department of Mathematics and Computer Science, Weyertal 86-90, 50931 Cologne, Germany}
\email{kbringma@math.uni-koeln.de}
\author{Jeremy Lovejoy}
\address{CNRS, Universit{\'e} Paris Cit\'e, B\^atiment Sophie Germain, Case courier 7014,
8 Place Aur\'elie Nemours, 75205 Paris Cedex 13, FRANCE}
\email{lovejoy@math.cnrs.fr}

\date{\today}
\subjclass[2020]{11F03, 11F37, 11P82, 11P83, 33D15}
\keywords{asymptotics, Euler--Maclaurin summation formula, mock theta functions, modular forms, partitions, Tauberian theorems, unimodal sequences}

\begin{document}

\begin{abstract}
	In this paper we study odd unimodal and odd strongly unimodal sequences. We use $q$-series methods to find several fundamental generating functions. Employing the Euler--Maclaurin summation formula we obtain the asymptotic main term for both types of sequences.  We also find families of congruences modulo $4$ for the number of odd strongly unimodal sequences.
\end{abstract}
\maketitle

\section{Introduction and statement of results}

A sequence is {\it unimodal} if it is weakly increasing up to a point and then weakly decreasing thereafter. Let $u(n)$ denote the number of unimodal sequences of natural numbers having the form 
\begin{equation}\label{unimodal}
	a_1 \leq \cdots \leq a_r \leq \overline{c} \geq b_1 \geq \cdots \geq b_s, 
\end{equation}
with
\[
	n = c + \sum_{j=1}^r a_j + \sum_{j=1}^s b_j.
\]
The distinguished point $c$ is called the {\it peak} of the sequence and the sum of the entries $n$ is called the {\it weight}. For example, we have $u(4)=12$, the twelve unimodal sequences of weight $4$ being
\begin{equation}\label{u(4)}
	\begin{gathered}
		\left(\ol4\right), \left(1,\ol3\right), \left(\ol3,1\right), \left(2,\ol2\right), \left(1,1,\ol2\right), \left(1,\ol2,1\right), \left(\ol2,2\right), \left(\ol2,1,1\right),\\
		\left(1,1,1,\ol1\right), \left(1,1,\ol1,1\right), \left(1,\ol1,1,1\right), \left(\ol1,1,1,1\right).
	\end{gathered}
\end{equation}

A unimodal sequence is {\it strongly unimodal} if the inequalities in \eqref{unimodal} are strict. Let $u^*(n)$ denote the number of strongly unimodal sequences of natural numbers with weight $n$. For example, we have $u^*(4)=4$, the four strongly unimdodal sequences of weight $4$ being\footnote{For strongly unimodal sequences we drop the overline notation since the peak cannot repeat.}
\begin{equation}\label{u^*(4)}
	(4), (1,3), (3,1), (1,2,1).
\end{equation}

Unimodal sequences and strongly unimodal sequences have been the subject of a considerable amount of research, especially over the last two decades. The generating functions for these sequences are related to number-theoretic objects like mock theta functions, false theta functions, and quantum modular forms, whose theories can then be applied to deduce many interesting results (see Subsections 14.4, 15.7, and 21.4 of \cite{Br-Fo-On-Ro}; for some more recent work, see \cite{Br-Na,Ch-Ga}).

In this paper we initiate the study of odd unimodal sequences and odd strongly unimodal sequences, wherein all numbers must be odd.\footnote{These were briefly treated in \cite{An2}, where they were called convex and strictly convex compositions with odd parts.} Let ${\rm ou}(n)$ and ${\rm ou}^*(n)$ denote the number of odd unimodal and odd strongly unimodal sequences of weight $n$. Continuing the example in \eqref{u(4)} and \eqref{u^*(4)}, we have 
${\rm ou}(4) = 6$ and ${\rm ou}^*(4) = 2$.

 We begin by computing the generating functions upon which all of our work is based. We use two-variable generating functions, where the second variable tracks the rank of the unimodal sequence. In the notation of \eqref{unimodal}, the rank is defined to be $r-s$, or the number of parts of the sequence to the left of the peak minus the number to the right. From the perspective of modular forms, this second variable is the ``Jacobi variable''. We use the standard $q$-hypergeometric notation, 
\begin{equation}\label{qPoch}
	(a;q)_n := \prod_{j=0}^{n-1} \left(1-aq^j\right),\qquad (a_1,\dots,a_\ell;q)_n := (a_1;q)_n\cdots(a_\ell;q)_n
\end{equation}
valid for $\ell\in\N$, $a,a_1,\dots,a_\ell\in\C$ and $n\in\N_0\cup\{\infty\}$. Let ${\rm ou}(m,n)$ denote the number of odd unimodal sequences of weight $n$ with rank $m$.

\begin{theorem}\label{T:umngf}
	We have
	\begin{align}\label{umngf1}
		&\sum_{\substack{n\ge0\\m\in\Z}} \textnormal{ou}(m,n)\z^mq^n = \sum_{n\ge0} \frac{q^{2n+1}}{\left(\z q,\z^{-1}q;q^2\right)_{n+1}}\\
		\label{umngf2}
		&\hspace{.99cm}= \sum_{n\ge0} (-1)^{n+1}\z^{3n+1}q^{3n^2+2n}\left(1+\z q^{2n+1}\right) + \frac{1}{\left(\z q,\z^{-1}q;q^2\right)_\infty}\sum_{n\ge0} (-1)^n\z^{2n+1}q^{n^2+n}\\
		\label{umngf3}
		&\hspace{.99cm}= \frac{q}{\left(q^2;q^2\right)_\infty}\left(\sum_{n,r\ge0}-\sum_{n,r<0}\right) \frac{(-1)^{n+r}q^{n^2+3n+4rn+r^2+3r}}{1-\z q^{2r+1}}.
	\end{align}
\end{theorem}

Let $\ou^*(m,n)$ denote the number of odd strongly unimodal sequences of weight $n$ with rank $m$.

\begin{theorem}\label{T:umnstargf}
	We have, with $Q(r,s):=\frac{r^2}{4}+\frac{7rs}{2}+\frac{s^2}{4}+\frac{3r}{2}+\frac{5s}{2}+1$,
	\begin{align}\label{umnstargf1}
		\sum_{\substack{n\ge0\\m\in\Z}} \ou^*(m,n)\z^mq^n &= \sum_{n\ge0} \left(-\z q,-\z^{-1}q;q^2\right)_nq^{2n+1}\\
		\label{umnstargf2}
		&= -\frac{1}{\left(q^2;q^2\right)_\infty}\sum_{n\in\Z} \frac{(-1)^nq^{3n^2+3n+1}}{1+\z q^{2n+1}} + \frac{1}{\left(q^2;q^2\right)_\infty}\sum_{n\in\Z} \frac{\z^{-n}q^{n^2+2n+1}}{1+\z q^{2n+1}}\\
		\label{umnstargf3}
		&= \frac{q}{\left(q^2;q^2\right)_\infty}\left(\sum_{n,r\ge0}-\sum_{n,r<0}\right) (-1)^n\z^rq^{3n^2+3n+4nr+r^2+2r}\\
		\label{umnstargf4}
		&= \frac{\left(-\z q,-\z^{-1}q;q^2\right)_\infty}{\left(q^2;q^2\right)_\infty^2}\left(\sum_{\substack{r,s\ge0\\r\equiv s\Pmod2}}-\sum_{\substack{r,s<0\\r\equiv s\Pmod2}}\right) \frac{(-1)^\frac{r-s}{2}q^{Q(r,s)}}{1+\z q^{r+s+1}}.
	\end{align}
\end{theorem}

\begin{remark}
	The analogue of \Cref{T:umngf} for ordinary unimodal sequences is \cite[Proposition 2.1]{Ki-Lo1}.  For strongly unimodal sequences, the corresponding generating functions are dispersed in the literature. The analogues of \eqref{umnstargf1} and \eqref{umnstargf2} are \cite[equation (14.18)]{Br-Fo-On-Ro} and \cite[Lemma 3.1]{Br-Fo-Rh}, the analogues of \eqref{umnstargf3} and \eqref{umnstargf4} are Theorems 4.1 and 1.3 of \cite{Hi-Lo1}. 
\end{remark}

Our first use of these generating functions is to find asymptotic estimates for ${\rm ou}(n)$ and ${\rm ou}^*(n)$.


\begin{theorem}\label{T:theo3}
	We have, as $n\to\infty$,
	\begin{equation} \label{ou(n)asymptotic}
		{\rm ou}(n) \sim \frac{e^{\pi\sqrt\frac{2n}{3}}}{2^\frac{13}{4}3^\frac14n^\frac34}.
	\end{equation}
\end{theorem}

\begin{theorem}\label{T:theo2}
	We have, as $n\to\infty$,
	\begin{equation} \label{ou^*(n)asymptotic}
		{\rm ou}^*(n) \sim \frac{e^{\pi\sqrt\frac n3}}{2^\frac523^\frac14n^\frac34}.
	\end{equation}
\end{theorem}


\begin{remark}
	By \cite{Au, Wr2}, the analogue of \eqref{ou(n)asymptotic} for the number of ordinary unimodal sequences $u(n)$ is 
	\begin{equation*}
		u(n) \sim \frac{e^{\pi\sqrt\frac{4n}{3}}}{2^3 3^\frac34 n^\frac54}.
	\end{equation*}
	The analogue of \eqref{ou^*(n)asymptotic} for the number of strongly unimodal sequences $u^*(n)$ is 
	\begin{equation} \label{u^*(n)asymptotic}
		u^*(n) \sim \frac{e^{\pi\sqrt\frac{2n}{3}}}{2^\frac{13}{4}3^\frac14n^\frac34},
	\end{equation}
	due to \cite{Rh}. Note that the asymptotics in \eqref{ou(n)asymptotic} and \eqref{u^*(n)asymptotic} agree.
\end{remark}

As a second result, we prove congruences modulo $4$ for the number of odd strongly unimodal sequences of weight $n$. Here we are motivated by a corresponding result for strongly unimodal sequences, which says that if $\ell\equiv7,11,13,17\Pmod{24}$ is prime and $(\frac j\ell)=-1$, then
\[
	u^*\left(\ell^2 n+\ell j-\left(\tfrac{\ell^2-1}{24}\right)\right) \equiv 0\Pmod4.
\]
This was conjectured by Bryson, Ono, Pitman, and Rhoades \cite{BOPR} and proved by Chen and Garvan \cite{Ch-Ga}. Our analogue for $\ou^*(n)$ is as follows; see \Cref{T:mod4congruences} for a more general theorem.

\begin{theorem}\label{T:mod4}
	Let $\ell\ge5$ be prime. If $\ell\equiv7,13\Pmod{24}$ and $(\frac{3j}{\ell})=-1$ or if $\ell\not\equiv7,13\Pmod{24}$ and $\ell\nmid j$, then we have:
	\begin{enumerate}
		\item If $j$ is odd, then
		\[
			\ou^*\left(4\ell^2n+2j\ell+\left(\tfrac{8\ell^2+1}{3}\right)\right) \equiv 0\Pmod4.
		\]
		
		\item If $j$ is even, then
		\begin{equation*}
			\ou^*\left(4\ell^2n+2j\ell+\left(\tfrac{2\ell^2+1}{3}\right)\right) \equiv 0\Pmod4.
		\end{equation*}
	\end{enumerate} 
\end{theorem}
 
The paper is organized as follows: In Section 2, we gather some necessary background on asymptotic methods and indefinite theta functions. In Section 3 we prove Theorems \ref{T:umngf} and \ref{T:umnstargf}. Sections 4 and 5 contain proofs of Theorems \ref{T:theo3} and \ref{T:theo2}. In Section 6, we show Theorem \ref{T:mod4congruences}, which contains the congruences in Theorem \ref{T:mod4} as a special case. We close in Section 7 with some open problems.

\section*{Acknowledgements}

The authors thank Caner Nazaroglu for help with numerical calculations. The first author has received funding from the European Research Council (ERC) under the European Union's Horizon 2020 research and innovation programme (grant agreement No. 101001179).

\section{Preliminaries}

\subsection{A Tauberian Theorem}

Recall the following\footnote{The second condition is often dropped in \eqref{E:as1} which makes the proposition unfortunately incorrect (see \cite{BJM}).} special case $\a=0$ of Theorem 1.1 of \cite{BJM}, which follows from Ingham's Theorem \cite{I}.

\begin{proposition}\label{P:CorToIngham}
	Suppose that $B(q)=\sum_{n\ge0}b(n)q^n$ is a power series with non-negative real coefficients and radius of convergence at least one and that the $b(n)$ are weakly increasing. Assume that $\l$, $\b$, $\g\in\R$ with $\g>0$ exist such that
	\begin{equation}\label{E:as1}
		B\left(e^{-t}\right) \sim \l t^\b e^\frac\g t \quad\text{as } t \to 0^+,\qquad B\left(e^{-z}\right) \ll |z|^\b e^\frac{\g}{|z|} \quad\text{as } z \to 0,
	\end{equation}
	with $z=x+iy$ ($x,y\in\R,x>0$) in each region of the form $|y|\le\De x$ for $\De>0$. Then
	\[
		b(n) \sim \frac{\l\g^{\frac\b2+\frac14}}{2\sqrt\pi n^{\frac\b2+\frac34}}e^{2\sqrt{\g n}} \qquad\text{as } n \to \infty.
	\]
\end{proposition}

\subsection{The Euler--Maclaurin summation formula}\label{SS:EMsf}

For simplicity we only state the versions of the Euler--Maclaurin summation formula that we use in this paper; see \cite{BJM} for a more general version for all dimensions. Let $D_\th:=\{re^{i\a}:r\ge0\text{ and }|\a|\le\th\}$. A multivariable function $f$ in $\ell$ variables is of {\it sufficient decay} in $D$ if there exist $\e_1,\dots,\e_\ell>0$ such that (we write vectors in bold letters) $f(\bm x)\ll(x_1+1)^{-1-\e_1}\cdots(x_\ell+1)^{-1-\e_\ell}$ uniformly as $|x_1|+\ldots+|x_\ell|\to\infty$ in $D$. We first require a one-dimensional version of the Euler--Maclaurin summation formula (see \cite{Za}).

\begin{proposition}\label{P:EulerMaclaurin1DShifted}
	Suppose that $0\le\th<\frac\pi2$. Let $f:\C\to\C$ be holomorphic in a domain containing $D_\th$, so that in particular $f$ is holomorphic at the origin, and assume that $f$ and all of its derivatives are of sufficient decay. Then for $a\in\R$ and $N\in\N_0$, we have, uniformly as $z\to0$ in $D_\th$,
	\[
		\sum_{m\ge0} f((m+a)z) = \frac1z\int_0^\infty f(w) dw - \sum_{n=0}^{N-1} \frac{B_{n+1}(a)f^{(n)}(0)}{(n+1)!}z^n + O_N\left(z^N\right),
	\]
	where $B_n(x)$ denotes the $n$-th Bernoulli polynomial.
\end{proposition}

We also require the two-dimensional case of the Euler--Maclaurin summation formula (see \cite{BJM}).

\begin{proposition}\label{P:EulerMaclaurinGeneral}
	Suppose that $0\le\th_j<\frac\pi2$ for $1\le j\le 2$, and that $f:\C^2\to\C$ is holomorphic in a domain containing $D_{\bm\th}:=D_{\th_1}\times D_{\th_2}$. If $f$ and all of its derivatives are of sufficient decay in $D_{\bm\th}$, then for $\bm a\in\R^2$ and $N\in\N_0$ we have uniformly, as $w\to0$ in $D_{\bm\th}$,
	\begin{align*}
		\sum_{\bm m\in\N_0^2} f((\bm m+\bm a)z) &= \frac{1}{z^2}\int_0^\infty \int_0^\infty f(\bm w) dw_1dw_2 - \frac1z\sum_{n_1=0}^{N-1} \frac{B_{n_1+1}(a_1)}{(n_1+1)!}z^{n_1}\int_0^\infty f^{(n_1,0)}(0,w_2) dw_2\\
		&\qquad - \frac1z\sum_{n_2=0}^{N-1} \frac{B_{n_2+1}(a_2)}{(n_2+1)!}z^{n_2}\int_0^\infty f^{(0,n_2)}(w_1,0) dw_1\\
		&\qquad + \sum_{n_1+n_2<N} \frac{B_{n_1+1}(a_1)B_{n_2+1}(a_2)f^{(n_1,n_2)}(\bm0)}{(n_1+1)!(n_2+1)!}z^{n_1+n_2} + O_N\left(z^N\right).
	\end{align*}
\end{proposition}

\subsection{Indefinite theta functions}

In this subsection, we recall results from Zwegers' thesis \cite{Zw}. Fix a quadratic form $Q$ of signature $(n,1)$ with associated matrix $A$, so that $Q(\bm x) = \frac{1}{2}\bm{x}^TA\bm{x}$. Let $B(\bm{x},\bm{y}) := Q(\bm{x} + \bm{y}) - Q(\bm{x}) - Q(\bm{y})$ denote the corresponding bilinear form. The set of vectors $\bm c\in\R^\ell$ with $Q(\bm c)<0$ splits into two connected components. Two vectors $\bm{c_1}$ and $\bm{c_2}$ lie in the same component if and only if $B(\bm{c_1},\bm{c_2})<0$. We fix one of the components and denote it by $C_Q$. Picking any vector $\bm{c_0}\in C_Q$, we have
\[
	C_Q = \left\{\bm c\in\R^\ell : Q(\bm c)<0,\ B(\bm c,\bm{c_0})<0\right\}.
\]
The cusps are elements from
\[
	S_Q := \left\{\bm c\in\Z^\ell : \gcd(c_1,c_2,\dots,c_\ell)=1,\ Q(\bm c)=0,\ B(\bm c,\bm{c_0})<0\right\}.
\]
Let $\ol C_Q:=C_Q\cup S_Q$ and define for, $\bm c\in\ol C_Q$
\[
	R(\bm c) :=
	\begin{cases}
		\R^\ell & \text{if }\bm c\in C_Q,\\
		\left\{\bm a\in\R^\ell:B(\bm c,\bm a)\notin\Z\right\} & \text{if }\bm c\in S_Q.
	\end{cases}
\]
Let $\bm{c_1},\bm{c_2}\in\ol C_Q$. We define the {\it theta function with characteristic} $\bm a\in R(\bm{c_1})\cap R(\bm{c_2})$ and $\bm b\in\R^\ell$ by
\begin{equation*}
	\vth_{\bm a,\bm b}(\t) := \sum_{\bm n\in\Z^\ell+\bm a} \p(\bm n;\t)e^{2\pi iB(\bm b,\bm n)}q^{Q(\bm n)},
\end{equation*}
where
\begin{equation*}
	\p(\bm n;\t) = \p_Q^{\bm{c_1},\bm{c_2}}(\bm n;\t) := \p^{\bm{c_1}}(\bm n;\t) - \p^{\bm{c_2}}(\bm n;\t)
\end{equation*}
with ($\t=u+iv$)
\[
	\p^{\bm c}(\bm n;\t) :=
	\begin{cases}
		E\left(\frac{B(\bm c,\bm n)\sqrt v}{\sqrt{-Q(\bm c)}}\right) & \text{if }\bm c\in C_Q,\\
		\sgn(B(\bm c,\bm n)) & \text{if }\bm c\in S_Q.
	\end{cases}
\]
Here the odd function $E$ is defined as
\[
	E(w) := 2\int_0^w e^{-\pi t^2} dt
\]
with the usual convention that $\sgn(w):=\frac{w}{|w|}$ for $w\in\R\sm\{0\}$ and $\sgn(0):=0$. Note that
\begin{equation}\label{E:EB}
	E(x) = \sgn(x)\left(1-\b\left(x^2\right)\right), \text{ where } \b(x) := \int_x^\infty w^{-\frac12}e^{-\pi w} dw.
\end{equation}
This in particular yields that $E(x)\sim\sgn(x)$ as $|x|\to\infty$.

The theta function satisfies the following transformation law.

\begin{theorem}\label{T:vthab}
	If $\bm a,\bm b\in R(\bm{c_1})\cap R(\bm{c_2})$, then
	\[
		\vth_{\bm a,\bm b}\left(-\frac1\t\right) = \frac{1}{\sqrt{-\det(A)}}(-i\t)^\frac\ell2e^{2\pi iB(\bm a,\bm b)}\sum_{\bm\ell\in A^{-1}\Z^\ell/\Z^\ell} \vth_{\bm b+\bm\ell,-\bm a}(\t).
	\]
\end{theorem}

\section{Generating functions and the proofs of Theorems \ref{T:umngf} and \ref{T:umnstargf}}

In this section we establish in Theorems \ref{T:umngf} and \ref{T:umnstargf}. We begin with \Cref{T:umngf}.

\begin{proof}[Proof of \Cref{T:umngf}]
	Equation \eqref{umngf1} is a straightforward consequence of the fact that $(\z q;q^2)_{n+1}^{-1}$ is the generating function for partitions into odd parts of size at most $2n+1$, with the exponent of $\z$ counting the number of parts. Namely, in the notation of \eqref{unimodal}, the term $q^{2n+1}$ generates the peak $\overline{c}$, the term $(\z q;q^2)_{n+1}^{-1}$ generates the odd parts $(a_1,\dots,a_r)$ to the left of the peak, and the term $( \z^{-1}q;q^2)_{n+1}^{-1}$ generates the odd parts $(b_1,\dots,b_s)$ to the right of the peak. The exponent of $\z$ is $r-s$. Equation \eqref{umngf2} is a result in Ramanujan's lost notebook \cite[Entry 6.3.4]{An-Be1}.
	
	Equation \eqref{umngf3} requires more work. Fot its proof, we require so-called Bailey pairs (for background see \cite{Mc}). A pair of sequences $(\a_n,\b_n)_{n\ge0}$ is called a {\it Bailey pair relative to $(a,q)$} if
	\begin{equation*}
		\beta_n = \sum_{k=0}^n \frac{\alpha_k}{(q;q)_{n-k}(aq;q)_{n+k}}.
	\end{equation*}
	If $(\a_n,\b_n)$ is a Bailey pair relative to $(a,q)$, then by \cite[equation (1.5)]{L}
	\begin{equation}\label{umnBaileylemma}
		\sum_{n\ge0} q^n\b_n = \frac{1}{(aq,q;q)_\infty}\sum_{n,r\ge0} (-a)^nq^{\frac{n(n+1)}{2}+(2n+1)r}\a_r.
	\end{equation}
	The following sequences form a Bailey pair relative to $(q^2,q^2)$ \cite[pp. 727--728]{Ki-Lo2}:
	\begin{equation}\label{umn}
		\a_n = \frac{(-1)^nq^{n^2+n}\left(1-q^{4n+2}\right)}{\left(1-q^2\right)\left(1-\z q^{2n+1}\right)\left(1-\z^{-1}q^{2n+1}\right)}, \quad \b_n = \frac{1}{\left(\z q,\z^{-1}q;q^2\right)_{n+1}}.
	\end{equation}
	Inserting \eqref{umn} into \eqref{umnBaileylemma} and using the fact that
	\begin{equation}\label{parfrac}
		\frac{1-q^{4r+2}}{\left(1-\z q^{2r+1}\right)\left(1-\z^{-1}q^{2r+1}\right)} = \frac{1}{1-\z q^{2r+1}} + \frac{\z^{-1}q^{2r+1}}{1-\z^{-1}q^{2r+1}},
	\end{equation}
	we compute
	\begin{align*}
		&\sum_{n\ge0} \frac{q^{2n+1}}{\left(\z q,\z^{-1}q;q^2\right)_{n+1}} = \frac{q}{\left(q^2;q^2\right)_\infty^2}\sum_{n,r\ge0} \frac{(-1)^{n+r}q^{n^2+3n+4nr+r^2+3r}\left(1-q^{4r+2}\right)}{\left(1-\z q^{2r+1}\right)\left(1-\z^{-1}q^{2r+1}\right)}\\
		&\hspace{1cm}= \frac{q}{\left(q^2;q^2\right)_\infty^2}\left(\sum_{n,r\ge0} \frac{(-1)^{n+r}q^{n^2+3n+4nr+r^2+3r}}{1-\z q^{2r+1}} + \z^{-1}\sum_{n,r\ge0} \frac{(-1)^{n+r}q^{n^2+3n+4nr+r^2+5r+1}}{1-\z^{-1}q^{2r+1}}\right)\\
		&\hspace{1cm}= \frac{q}{\left(q^2;q^2\right)_\infty^2}\left(\sum_{n,r\ge0}-\sum_{n,r<0}\right) \frac{(-1)^{n+r}q^{n^2+3n+4rn+r^2+3r}}{1-\z q^{2r+1}}.
	\end{align*}
	In the last step we let $(n,r)\mapsto(-n-1,-r-1)$ and simplify. This completes the proof.
\end{proof}

We now turn to the proof of Theorem \ref{T:umnstargf}.

\begin{proof}[Proof of Theorem \ref{T:umnstargf}]
	We use the fact that $(-\z q;q^2)_n$ is the generating function for partitions into distinct odd parts of size at most $2n-1$, with the exponent of $\z$ counting the number of parts.

	For \eqref{umnstargf2} we require two identities,
	\begin{align}
		\sum_{n\ge0} \frac{q^{2n^2+2n+1}}{\left(-\z q,-\z^{-1}q;q^2\right)_{n+1}} &= \frac{1}{\left(q^2;q^2\right)_\infty}\sum_{n\in\Z} \frac{(-1)^nq^{3n^2+3n+1}}{1+\z q^{2n+1}}\label{omega},\\
		\sum_{n\in\Z} \frac{(a,b;q)_nw^n}{(c,d;q)_n} &= \frac{\left(aw,\frac da,\frac cb,\frac{dq}{abw};q\right)_\infty}{\left(w,d,\frac qb,\frac{cd}{abw};q\right)_\infty}\sum_{n\in\Z} \frac{\left(a,\frac{abw}{d};q\right)_n\left(\frac da\right)^n}{(aw,c;q)_n}.
		\label{2psi2}
	\end{align}
	\Cref{omega} may be found in \cite[p. 397]{ABL}, while equation \eqref{2psi2} is a bilateral transformation of Bailey \cite[example 5.20 (i)]{GR}. The notation in \eqref{qPoch} is extended to all integers via
	\begin{equation*}
		(a;q)_n = \frac{(a;q)_{\infty}}{(aq^n;q)_{\infty}}.
	\end{equation*} 
	Note that we have the identity \cite[(I.2)]{GR}
	\begin{equation}\label{qPoch-n}
		(a;q)_{-n} = \frac{(-1)^nq^\frac{n(n+1)}{2}}{a^n\left(\frac qa;q\right)_n}.
	\end{equation}

	We begin by letting $(a,b,w,q)=(-\z q,-\z^{-1}q,q^2,q^2)$ in \eqref{2psi2} and then letting $c,d\to0$. Simplifying and exchanging left- and right-hand sides gives
	\begin{align*}
		&\frac{1}{\left(q^2;q^2\right)_\infty}\sum_{n\in\Z} \frac{\z^{-n}q^{n^2+2n+1}}{1+\z q^{2n+1}} = \sum_{n\in\Z} \left(-\z q,-\z^{-1}q;q^2\right)_nq^{2n+1}\\
		&\hspace{.3cm}= \sum_{n\ge0} \left(-\z q,-\z^{-1}q;q^2\right)_nq^{2n+1} + \sum_{n\le-1} \left(-\z q,-\z^{-1}q;q^2\right)_nq^{2n+1}\\
		&\hspace{.3cm}= \sum_{\substack{n\ge0\\m\in\Z}} \ou^*(m,n)\z^mq^n + \sum_{n\ge1} \left(-\z q,-\z^{-1}q;q^2\right)_{-n}q^{-2n+1}\\
		&\hspace{.3cm}= \sum_{\substack{n\ge0\\m\in\Z}} \ou^*(m,n)\z^mq^n + \sum_{n\ge1} \frac{q^{2n^2-2n+1}}{\left(-\z q,-\z^{-1}q;q^2\right)_n} = \sum_{\substack{n\ge0\\m\in\Z}} \ou^*(m,n)\z^mq^n + \sum_{n\ge0} \frac{q^{2n^2+2n+1}}{\left(-\z q,-\z^{-1}q;q^2\right)_{n+1}}\\
		&\hspace{.3cm}= \sum_{\substack{n\ge0\\m\in\Z}} \ou^*(m,n)\z^mq^n + \frac{1}{\left(q^2;q^2\right)_\infty}\sum_{n\in\Z} \frac{(-1)^nq^{3n^2+3n+1}}{1+\z q^{2n+1}}.
	\end{align*}
	Here the final equality uses \eqref{omega} and the antepenultimate equality uses \eqref{qPoch-n}. Comparing the extremes in this string of equations gives \eqref{umnstargf2}. 

	We now turn to \eqref{umnstargf3}. For this require the fact that if $(\a_n,\b_n)$ is a Bailey pair relative to $(a,q)$, then we use \cite[Corollary 1.3]{L}
	\begin{equation}\label{umnstarBaileylemma}
		\sum_{n\ge0} (aq;q)_{2n}q^n\b_n = \frac{1}{(q;q)_\infty}\sum_{n,r\ge0} (-a)^nq^{\frac{3n(n+1)}{2}+(2n+1)r}\a_r,
	\end{equation}
	along with the following Bailey pair relative to $(1,q)$ \cite[Lemma 3]{An1}:
	\[
		\a_n =
		\begin{cases}
			(-1)^n\left(w^nq^\frac{n(n-1)}{2}+w^{-n}q^\frac{n(n+1)}{2}\right) & \text{if }n\ge1,\\
			1 & \text{if }n=0,
		\end{cases}
		\qquad \b_n = \frac{\left(w;\frac qw;q\right)_n}{(q;q)_{2n}}.
	\]
	Using this Bailey pair with $(w,q)=(-\z q,q^2)$ in \eqref{umnstarBaileylemma} we compute
	\begin{align*}
		&\sum_{n\ge0} \left(-\z q,-\z^{-1}q;q^2\right)_nq^{2n+1}\\
		&\hspace{2.4cm}= \frac{q}{\left(q^2;q^2\right)_\infty}\left({\vphantom{\sum_{\substack{n\ge0\\r\ge1}}}}\right.\sum_{n,r\ge0} (-1)^n\z^rq^{3n^2+3n+4nr+r^2+2r}+\sum_{\substack{n\ge0\\r\ge1}} (-1)^n\z^{-r}q^{3n^2+3n+4nr+r^2+2r}\left.{\vphantom{\sum_{\substack{n\ge0\\r\ge1}}}}\right)\\
		&\hspace{2.4cm}= \frac{q}{\left(q^2;q^2\right)_\infty}\left(\sum_{n,r\ge0} (-1)^n\z^rq^{3n^2+3n+4nr+r^2+2r}-\sum_{n,r<0} (-1)^n\z^rq^{3n^2+3n+4nr+r^2+2r}\right),
	\end{align*} 
	where the last line follows upon replacing $(n,r)$ by $(-n-1,-r)$. This gives \eqref{umnstargf3}.

	Finally, we treat \eqref{umnstargf4}. Again we use Bailey pairs. This time we require the fact that if $(\a,\b_n)$ is a Bailey pair relative to $(a,q)$, then \cite[Theorem 10.1]{Mc}
	\begin{equation}\label{umnstarBaileylemma2}
		\sum_{n\ge0} (b,c;q)_n\left(\frac{aq}{bc}\right)^n\b_n = \frac{\left(\frac{aq}{b},\frac{aq}{c};q\right)_\infty}{\left(aq,\frac{aq}{bc};q\right)_\infty} \sum_{n\ge0} \frac{(b,c;q)_n\left(\frac{aq}{bc}\right)^n}{\left(\frac{aq}{b},\frac{aq}{c};q\right)_n}\a_n,
	\end{equation}
	along with the Bailey pair\footnote{We point out to the reader that the $A_n$ in Andrews' paper are related to the $\alpha_n$ via $\alpha_n = a^nq^{n^2}A_n$.} relative to $(q,q)$ from \cite[equation (5.11)]{An},
	\[
		\a_n = \frac{q^{2n^2+n}\left(1-q^{2n+1}\right)}{1-q}\sum_{j=-n}^n (-1)^jq^{-\frac{j(3j+1)}{2}},\qquad \b_n = 1.
	\]
	Using this Bailey pair from \eqref{umnstarBaileylemma2} with $(b,c,q)=(-\z q,-\z^{-1}q,q^2)$ and employing \eqref{parfrac}, we compute
	\begin{align*}
		&\sum_{n\ge0} \left(-\z q,-\z^{-1}q;q^2\right)_nq^{2n+1}\\
		&\hspace{.2cm}= \frac{q\left(-\z q,-\z^{-1}q;q^2\right)_\infty}{\left(q^2;q^2\right)_\infty^2}\sum_{n\ge0} \frac{q^{4n^2+4n}\left(1-q^{4n+2}\right)}{\left(1+\z q^{2n+1}\right)\left(1+\z^{-1}q^{2n+1}\right)}\sum_{j=-n}^n (-1)^jq^{-j(3j+1)}\\
		&\hspace{.2cm}= \frac{q\left(-\z q,-\z^{-1}q;q^2\right)_\infty}{\left(q^2;q^2\right)_\infty^2}\left(\sum_{n\ge0} \sum_{j=-n}^n \frac{(-1)^jq^{4n^2+4n-j(3j+1)}}{1+\z q^{2n+1}}-\z^{-1}\sum_{n\ge0} \sum_{j=-n}^n \frac{(-1)^jq^{4n^2+6n+1-j(3j+1)}}{1+\z^{-1}q^{2n+1}}\right)\\
		&\hspace{.2cm}= \frac{q\left(-\z q,-\z^{-1}q;q^2\right)_\infty}{\left(q^2;q^2\right)_\infty^2} \left(\sum_{n\ge0} \sum_{j=-n}^n \frac{(-1)^jq^{4n^2+4n-j(3j+1)}}{1+\z q^{2n+1}} - \sum_{n\le0} \sum_{j=n}^{-n} \frac{(-1)^jq^{4n^2-4n-j(3j+1)}}{1+\z q^{2n-1}}\right).
	\end{align*}
	Now letting $(n,j)=(\frac{r+s}{2},\frac{r-s}{2})$ in the first sum on the right-hand side, letting $(n,j)=(\frac{r+s+2}{2},\frac{r-s}{2})$ in the second sum, and then simplifying gives \eqref{umnstargf4}. This completes the proof of \Cref{T:umnstargf}.
\end{proof}

\section{Proof of \Cref{T:theo3}}

In this section, we prove \Cref{T:theo3}.

\begin{proof}[Proof of \Cref{T:theo3}]
	From Theorem \ref{T:umngf} we have that
	\begin{equation}\label{ou(n)partialtheta}
		\sum_{n\ge0} \ou(n)q^n = \sum_{n\ge0} \frac{q^{2n+1}}{\left(q;q^2\right)_{n+1}^2}		= \sum_{n\ge0} (-1)^{n+1}q^{n(3n+2)}\left(1+q^{2n+1}\right)+\frac{1}{\left(q;q^2\right)_\infty^2}\sum_{n\ge0} (-1)^nq^{n^2+n}.
	\end{equation}
	We apply \Cref{P:CorToIngham}. To begin, it is not hard to see that the ${\rm ou}(n)$ are monotonic, since
	\begin{equation*}
		(1-q)\sum_{n \geq 0} {\rm ou}(n)q^n = \sum_{n \geq 0} \frac{q^{2n+1}}{(q^3;q^2)_n(q;q^2)_{n+1}}, 
	\end{equation*}
	and the right-hand side has non-negative coefficients. Alternatively, note that for any $n\in\N$
	\begin{equation*}
		\left(a_1,\dots,a_r,\ol c,b_1,\dots,b_s\right) \mapsto \left(1,a_1,\dots,a_r,\ol c,b_1,\dots,b_s\right)
	\end{equation*}
	is an injective mapping from the set of odd unimodal sequences of weight $n$ to the set of odd unimodal sequences of weight $n+1$.

	Next, let $F_2(q)$ denote the right-hand side of \eqref{ou(n)partialtheta}. We analyze each of the terms separately with the goal of showing that as $z \to 0$,
	\[
		F_2\left(e^{-z}\right) \sim \frac{e^\frac{\pi^2}{6z}}4.
	\]
	The modularity of the Dedekind $\eta$-function implies that
	\begin{equation}\label{E:partas}
		\frac{1}{\left(e^{-z};e^{-z}\right)_\infty} \sim \sqrt\frac{z}{2\pi}e^\frac{\pi^2}{6z} \qquad\text{as } z \to 0.
	\end{equation}
	Thus, with $q=e^{-z}$, we have
	\begin{equation}\label{E:qe}
		\frac{1}{\left(q;q^2\right)_\infty^2} = \frac{\left(q^2;q^2\right)_{\infty}^2}{(q;q)_{\infty}^2} \sim \frac{e^\frac{\pi^2}{6z}}2.
	\end{equation}
	We apply \Cref{P:EulerMaclaurin1DShifted} to the two sums in \eqref{ou(n)partialtheta} . We start by splitting the first sum according to the parity of the summation variable in order to rewrite it as
	\[
		\sum_{n\ge0} (-1)^{n+1}q^{n(3n+2)}\left(1+q^{2n+1}\right) = q^{-\frac13}\sum_{n\ge0} \left(q^{12\left(n+\frac23\right)^2}+q^{12\left(n+\frac56\right)^2}-q^{12\left(n+\frac16\right)^2}-q^{12\left(n+\frac13\right)^2}\right).
	\]
	Now we can apply \Cref{P:EulerMaclaurin1DShifted} with $f(z):=e^{-12z^2}$ and $a\in\{\frac23,\frac56,\frac16,\frac13\}$. The main terms from \Cref{P:EulerMaclaurin1DShifted} cancel and using that $q^{-\frac13}=O(1)$ we obtain that the first sum is $O(1)$.
	
	For the second sum we write, using \Cref{P:EulerMaclaurin1DShifted}
	\begin{align*}
		\sum_{n\ge0} (-1)^ne^{-\left(n^2+n\right)z} &= \sum_{n\ge0} e^{-\left(4n^2+2n\right)z} - \sum_{n\ge0} e^{-\left((2n+1)^2+(2n+1)\right)z}\\
		&= e^\frac z4\left(\sum_{n\ge0} \left(e^{-4\left(n+\frac14\right)^2z}-e^{-4\left(n+\frac34\right)^2z}\right)\right) \sim -B_1\left(\tfrac14\right) + B_1\left(\tfrac34\right) = \tfrac12.
	\end{align*}
	Combining with \eqref{E:qe}, \Cref{P:CorToIngham} with $\l=\frac14$, $\a=\b=0$ gives the claim.
\end{proof}

\section{Proof of \Cref{T:theo2}}

In this section, we prove \Cref{T:theo2}. As in the previous section, we wish to apply \Cref{P:CorToIngham}, though in this case the details are much more involved. To begin, we record the monotonicity of the sequence $\ou^*(n)$.

\begin{lemma} \label{L:oustarmonotonic}
For $n \geq 3$ we have that ${\rm ou}^*(n) \geq {\rm ou}^*(n-1)$. 
\end{lemma}

\begin{proof}
	We give two proofs, one employing $q$-series and one using a combinatorial argument. For the $q$-series proof, first observe that
	\begin{align}\nonumber
		(1-q)\sum_{n\ge0} \ou^*(n)q^n &= (1-q)\sum_{n\ge0} \left(-q;q^2\right)_n^2q^{2n+1} = q(1-q) + (1-q)\sum_{n\ge1} \left(-q;q^2\right)_n^2q^{2n+1}\\
		\label{last}
		&= q(1-q) + q^3\left(1-q^2\right)(1+q)\sum_{n\ge0} \left(-q^3;q^2\right)_n^2q^{2n}.
	\end{align}
	We now require the $q$-binomial theorem \cite[Exercise 1.2]{GR}
	\begin{equation}\label{qbinomial}
		\sum_{m = 0}^n \frac{(q;q)_nw^nq^\frac{m(m-1)}{2}}{(q;q)_m(q;q)_{n-m}} = (-w;q)_n
	\end{equation} 
	and a transformation of Jackson \cite[Appendix (III.4)]{GR},
	\begin{equation}\label{Jackson}
		\sum_{n\ge0} \frac{(a,b;q)_nw^n}{(c,q;q)_n} = \frac{(aw;q)_\infty}{(w;q)_\infty}\sum_{n\ge0} \frac{\left(a,\frac cb;q\right)_n(-bw)^nq^\frac{n(n-1)}{2}}{(c,aw,q;q)_n}.
	\end{equation}
	Using these, we rewrite the final sum in \eqref{last} as follows:
	\begin{align}\nonumber
		&\sum_{n\ge0} \left(-q^3;q^2\right)_n^2q^{2n} = \sum_{n\ge0} \left(-q^3;q^2\right)_nq^{2n}\sum_{m=0}^n \frac{\left(q^2;q^2\right)_nq^{m^2+2m}}{\left(q^2;q^2\right)_m\left(q^2;q^2\right)_{n-m}}\\
		\nonumber
		&\hspace{.6cm}= \sum_{m\ge0} \frac{q^{m^2+2m}}{\left(q^2;q^2\right)_m}\sum_{n\ge m} \frac{\left(-q^3,q^2;q^2\right)_nq^{2n}}{\left(q^2;q^2\right)_{n-m}} = \sum_{m\ge0} q^{m^2+4m}\left(-q^3;q^2\right)_m\sum_{n\ge0} \frac{\left(-q^{2m+3},q^{2m+2};q^2\right)_nq^{2n}}{\left(q^2;q^2\right)_n}\\
		\label{last2}
		&\hspace{.6cm}= \sum_{m\ge0} \frac{q^{m^2+4m}\left(-q^3;q^2\right)_m}{\left(q^2;q^2\right)_m}\sum_{n\ge0} \frac{q^{n^2+4n+2nm}}{\left(q^2;q^2\right)_n\left(1-q^{2n+2m+2}\right)}.
	\end{align}
	Here the first equality follows from \eqref{qbinomial} and the final equality implied by \eqref{Jackson} with $(a,b,c,w,q)=(q^{2m+2},-q^{2m+3},0,q^2,q^2)$. Combining \eqref{last} and \eqref{last2} gives
	\[
		(1-q)\sum_{n\ge0} \ou^*(n)q^n = q(1-q) + q^2\left(1-q^2\right)(1+q)\sum_{n,m\ge0} \frac{q^{n^2+4n+m^2+4m+2nm}\left(-q^3;q^2\right)_m}{\left(q^2;q^2\right)_n\left(q^2;q^2\right)_m\left(1-q^{2n+2m+2}\right)}.
	\]
	It is straightforward to see that the coefficient of $q^n$ on the right-hand side is non-negative for $n \geq 3$.

	Alternatively, one may deduce the montonicity using a combinatorial argument. For $n\ge3$ we define a mapping on odd strongly unimodal sequences of weight $n$ as follows:
	\[
		\left(a_1,\dots,a_r,\ol c,b_1,\dots,b_s\right) \mapsto
		\begin{cases}
			\left(1,a_1,\dots,a_r,\ol c,b_1,\dots,b_s\right) & \text{if }a_1\ne1,\\
			\left(a_2,\dots,a_r,\ol{c+2},b_1,\dots,b_s\right) & \text{if }a_1=1.
		\end{cases}
	\]
	It is not hard to see that in either case the image is an odd strongly unimodal sequence of weight $n+1$ and that the mapping is injective. This gives the desired inequality ${\rm ou}^*(n) \geq {\rm ou}^*(n-1)$ for $n \geq 4$, and the case $n=3$ follows from the fact that ${\rm ou}^*(3) = 1$ and ${\rm ou}^*(2) = 0$.
\end{proof}

We now turn to the proof of Theorem \ref{T:theo2}.

\begin{proof}[Proof of \Cref{T:theo2}]
	\Cref{L:oustarmonotonic} implies that the ${\rm ou}^*(n)$ are monotonic. The rest of the proof is devoted to showing that the remaining conditions of \Cref{P:CorToIngham} are satisfied. By \eqref{umnstargf3}
\begin{equation*}
		\sum_{n \geq 0} {\rm ou}^*(n)q^n = \frac{q}{\left(q^2;q^2\right)_\infty}\left(\sum_{r,n\ge0}-\sum_{r,n<0}\right) (-1)^nq^{3n^2+4nr+r^2+3n+2r}.
\end{equation*}
Denoting by $F_1(q)$ the right-hand side, we aim to prove that
	\begin{equation}\label{E:as2}
		F_1\left(e^{-t}\right) \sim \frac{e^\frac{\pi^2}{12t}}2 \quad\text{as } t \to 0,\qquad F\left(e^{-z}\right) \ll e^\frac{\pi^2}{12|z|} \quad\text{as } z \to 0,
	\end{equation}
	with $z=x+iy$ ($x,y\in\R$, $x>0$, $|y|\le\De x$, $\De>0$).
	We first consider the outside factor. 
By \eqref{E:partas} with $q=e^{-z}$ we have, as $z\to0$,
	\begin{equation*}
		\frac{q}{\left(q^2;q^2\right)_\infty} \sim \sqrt\frac z\pi e^{\frac{\pi^2}{12z}}.
	\end{equation*}
	
	Next define
	\[
		G(q) := \frac12\sum_{n,r\in\Z} \left(\sgn\left(n+\tfrac12\right)+\sgn\left(r+\tfrac12\right)\right)(-1)^n q^{3n^2+4nr+r^2+3n+2r},
	\]
	so $F_1(q)=\frac{G(q)}{(q^2;q^2)_\infty}$. We realize $G$ as ``holomorphic part'' of an indefinite theta function. For this set
	\[
		g(\t) := 2q^\frac34G(q) = i\sum_{\bm n\in\Z^2+\bm a} (\sgn(B(\bm{c_1},\bm n))-\sgn(B(\bm{c_2},\bm n)))e^{2\pi iB(\bm b,\bm n)}q^{Q(\bm n)},
	\]
	where $Q(\bm n):=3n_1^2+4n_1n_2+n_2^2$, $\bm{c_1}:=(1,-2)^T$, $\bm{c_2}:=(2,-3)^T$, $\bm b:=(-\frac14,\frac12)^T$, and $\bm a:=(\frac12,0)^T$.
	
	Using \eqref{E:EB}, we may decompose
	\[
		g(\t) = \Th(\t) + \Th^-(\t),
	\]
	where
	\begin{align*}
		\Th(\t) &:= i\sum_{\bm n\in\Z^2+\bm a} \left(E\left(\tfrac{B(\bm{c_1},\bm n)}{\sqrt{-Q(\bm{c_1})}}\sqrt v\right)-E\left(\tfrac{B(\bm{c_2},\bm n)}{\sqrt{-Q(\bm{c_2})}}\sqrt v\right)\right)e^{2\pi iB(\bm b,\bm n)}q^{Q(\bm n)},\\
		\Th^-(\t) &:= \sum_{\bm n\in\Z^2+\bm a} \left(\sgn(n_1)\b\left(4n_1^2v\right)+\sgn(n_2)\b\left(\tfrac{4n_2^2v}{3}\right)\right) (-1)^{n_1-\frac12}q^{3n_1^2+4n_1n_2+n_2^2}.
	\end{align*}
	In fact the identity holds termwise and we use that
	\begin{align*}
		3n_1^2 + 4n_1n_2 + n_2^2 &= Q(\bm n),\quad (-1)^{n_1-\frac12} = -ie^{2\pi iB(\bm b,\bm n)},\quad n_1 = -\tfrac12B(\bm{c_1},\bm n),\quad n_2 = \tfrac12B(\bm{c_2},\bm n),\\
		Q(\bm{c_1}) &= -1,\qquad Q(\bm{c_2}) = -3.
	\end{align*}
	We determine the asymptotic behavior of $\Theta$ and $\Theta^-$ separately: For $\Th$ we use modularity and for $\Th^-$ the Euler--Maclaurin summation formula.
	
	We start with $\Th$. We have that (in the notation of Subsection \ref{SS:EMsf})
	\[
		\Th(\t) = i\vth_{\bm a,\bm b}(\t).
	\]
	We apply \Cref{T:vthab} to obtain
	\[
		\vth_{\bm a,\bm b}(\t) = -\frac{1}{2\t}\sum_{\bm\ell\in\left\{\bm0,\left(\frac12,0\right),\left(0,\frac12\right), \left(\frac12,\frac12\right)\right\}} \vth_{\bm b+\bm\ell,-\bm a}\left(-\frac1\t\right).
	\]
	Now write
	\begin{equation}\label{trep}
		\vth_{\bm b+\bm\ell,-\bm a}(\t) = -\sum_{\bm n\in\Z^2+\bm b+\bm\ell} \left(E\left(2n_1\sqrt v\right)+E\left(\tfrac{2n_2\sqrt v}{\sqrt3}\right)\right)e^{2\pi iB(-\bm a,\bm n)}q^{Q(\bm n)}.
	\end{equation}
	Using that $E(x)\sim\sgn(x)$ as $|x|\to\infty$ the terms in \eqref{trep} exponentially decay.
	
	We next investigate $\Th^-$ and write it as
	\begin{align*}
		\Th^-(\t) &= \sum_{\bm n\in\Z^2+\bm a} \left(\sgn(n_1)\b\left(4n_1^2v\right)+\sgn(n_2)\b\tleg{4n_2^2v}{3}\right) (-1)^{n_1-\tfrac12}q^{3n_1^2+4n_1n_2+n_2^2}\\
		&= \sum_{\bm n\in\Z^2} \left(\sgn\left(n_1+\tfrac12\right)\b\left(4\left(n_1+\tfrac12\right)^2v\right) + \sgn(n_2)\b\tleg{4n_2^2v}{3}\right)(-1)^{n_1}\\
		&\hspace{9.5cm}\times q^{3\left(n_1+\tfrac12\right)^2+4\left(n_1+\tfrac12\right)n_2+n_2^2}\\
		&= \sum_{\substack{n_1\ge0\\n_2\ge1}} (-1)^{n_1} \left(\b\left(4\left(n_1+\tfrac12\right)^2v\right)+\b\left(\tfrac43n_2^2v\right)\right) q^{3\left(n_1+\tfrac12\right)^2+4\left(n_1+\tfrac12\right)n_2+n_2^2}\\
		&+ \sum_{\substack{n_1\ge0\\n_2\ge1}} (-1)^{-n_1-1} \left(-\b\left(4\left(-n_1-1+\tfrac12\right)^2v\right)+\b\left(\tfrac43n_2^2v\right)\right) q^{3\left(-n_1-1+\tfrac12\right)^2+4\left(-n_1-1+\tfrac12\right)n_2+n_2^2}\\
		&+ \sum_{\substack{n_1\ge0\\n_2\ge1}} (-1)^{n_1} \left(\b\left(4\left(n_1+\tfrac12\right)^2v\right)-\b\left(\tfrac43(-n_2)^2v\right)\right) q^{3\left(n_1+\tfrac12\right)^2+4\left(n_1+\tfrac12\right)(-n_2)+(-n_2)^2}\\
		&+ \sum_{n_1\ge0} (-1)^{-n_1-1} \left(-\b\left(4\left(-n_1-1+\tfrac12\right)^2v\right)-\b\left(\tfrac43(-n_2)^2v\right)\right)\\
		&\hspace{7cm}\times q^{3\left(-n_1-1+\tfrac12\right)^2+4\left(-n_1-1+\tfrac12\right)(-n_2)+(-n_2)^2}\\
		&= 2\sum_{\substack{n_1\ge0\\n_2\ge1}} (-1)^{n_1} \left(\b\left(4\left(n_1+\tfrac12\right)^2v\right)+\b\tleg{4n_2^2v}{3}\right) q^{3\left(n_1+\tfrac12\right)^2+4\left(n_1+\tfrac12\right)n_2+n_2^2}\\
		&+ 2\sum_{\substack{n_1\ge0\\n_2\ge1}} (-1)^{n_1} \left(\b\left(4\left(n_1+\tfrac12\right)^2v\right)-\b\tleg{4n_2^2v}{3}\right) q^{3\left(n_1+\tfrac12\right)^2-4\left(n_1+\tfrac12\right)n_2+n_2^2}\\
		&= 2\sum_\pm \sum_{\d\in\{0,1\}} (-1)^\d\sum_{n_1,n_2\ge0} \left(\b\left(16\left(n_1+\tfrac\d2+\tfrac14\right)^2v\right) \pm \b\left(\tfrac43(n_2+1)^2v\right)\right)\\
		&\hspace{7cm}\times q^{12\left(n_1+\tfrac\d2+\tfrac14\right)^2 \pm 8\left(n_1+\tfrac\d2+\tfrac14\right)(n_2+1)+(n_2+1)^2}.
	\end{align*}
	We now first show the first asymptotic in \eqref{E:as2}. For this, let $\t=\frac{it}{2\pi}$. Then
	\[
		\Th^-\left(\tfrac{it}{2\pi}\right) = 2\sum_\pm \sum_{\d\in\{0,1\}} (-1)^\d\sum_{n_1,n_2\ge0} f_\pm \left(\left(n_1+\tfrac\d2+\tfrac14,n_2+1\right)\sqrt t\right),
	\]
	where
	\[
		f_\pm(x_1,x_2) := \left(\b\left(\tfrac{8x_1^2}{\pi}\right)\pm\b\left(\tfrac{2x_2^2}{3\pi}\right)\right)e^{-12x_1^2\mp8x_1x_2-x_2^2}.
	\]

	We now use \Cref{P:EulerMaclaurinGeneral}. The term with the double integral term vanishes (the two $\d$-terms cancel). The second term contributes 
	\[
		-\frac{2}{\sqrt t}\sum_\pm \sum_{\d\in\{0,1\}} (-1)^\d\sum_{n_1=0}^{N-1} \frac{B_{n_1+1}\left(\frac\d2+\frac14\right)t^\frac{n_1}{2}}{(n_1+1)!}\int_0^\infty f_\pm^{(n_1,0)}(0,w_2) dw_2.
	\]
	By combining the term for $\d=0$ and $\d=1$, using properties of Bernoulli polynomials it is not hard to see that only $n_1$ even survive. The terms from $n_1\ge1$ yield a contribution overall that is $O(t)$. Using that $\b(0)=1$, the term $n_1=0$ gives
	\[
		-\frac{4}{\sqrt t}\sum_\pm B_1\tleg14\int_0^\infty \left(1\pm\b\left(\tfrac{2x_2^2}{3\pi}\right)\right)e^{-x_2^2} dx_2 = \frac{2}{\sqrt t}\int_0^\infty e^{-w^2_2} dw_2 = \sqrt{\tfrac\pi t}.
	\]
	
	For the third term, we have
	\[
		-\frac{2}{\sqrt t}\sum_\pm \sum_{\d\in\{0,1\}} (-1)^\d\sum_{n_2=0}^{N-1} \frac{B_{n_2+1}(1)}{(n_2+1)!}t^\frac{n_2}{2}\int_0^\infty f_\pm^{(0,n_2)}(w_1,0) dw_1 = 0,
	\]
	because the $\d$-terms cancel. The final term in \Cref{P:EulerMaclaurinGeneral} is in $O(t)$. This gives that the first asymptotic in \eqref{E:as2} holds.
	
	We next need to show that the second asymptotic in \eqref{E:as2} holds. For this, we need to prove that
	\begin{multline}\label{E:second2}
		\sum_\pm \sum_{\d\in\{0,1\}} (-1)^\d\sum_{n_1,n_2\ge0} \left(\b\left(\tfrac8\pi\left(n_1+\tfrac\d2+\tfrac14\right)^2x\right) \pm \b\left(\tfrac{2}{3\pi}(n_2+1)^2x\right)\right)\\
		\times e^{-\left(12\left(n_1+\frac\d2+\frac14\right)^2 \pm 8\left(n_1+\frac\d2+\frac14\right)(n_2+1)+(n_2+1)^2\right)z} \ll |z|^\frac12.
	\end{multline}
	The proof follows by a lengthy calculation from the following refinement of \Cref{P:EulerMaclaurinGeneral} in the one-dimensional case, namely (see (5.8) of \cite{BJM})
	\begin{equation}\label{E:refine}
		\sum_{n\ge0} f((n+a)z) = \frac1z\int_0^\infty f(w) dw - \sum_{n=0}^{N-1} \frac{B_{n+1}(a)f^{(n)}(0)}{(n+1)!}z^n + \Ec(a;z),
	\end{equation}
	where ($C_R(0)$ is the circle around $0$ with radius $R$, $\wt B_N(x):=B_N(x-\flo x)$)
	\begin{multline*}
		\Ec(a;z) := -\sum_{k\ge N} \frac{f^{(k)}(0)a^{k+1}}{(k+1)!}z^k - \frac{z^N}{2\pi i}\sum_{n=0}^{N-1} \frac{B_{n+1}(0)a^{N-n}}{(n+1)!}\int_{C_R(0)} \frac{f^{(n)}(w)}{w^{N-n}(w-az)} dw\\
		-\frac{(-1)^Nz^{N-1}}{N!}\int_{az}^{z\infty} f^{(N)}(w)\wt B_N\left(\frac wz-a\right) dw.
	\end{multline*}
	For the reader's convenience we defer the full proof of \eqref{E:second2} to Appendix A.

	Combining and using \Cref{P:CorToIngham} with $\g=\frac{\pi^2}{12}$, $\b=0$, and $\l=\frac12$ gives the claim.
\end{proof}

\section{Congruences for $\ou^*(n)$ modulo $4$ and the proof of \Cref{T:mod4}}

In this section we prove Theorem \ref{T:mod4congruences} below.   Note that this reduces to Theorem \ref{T:mod4} for $k=0$.   First, we determine the parity of ${\rm ou}^*(n)$.

\begin{proposition}\label{P:oun}
	For $n\in\N$, we have that $\ou^*(n)$ is odd if and only if $6n-2$ is a square.
\end{proposition}

\begin{proof}
	We require a classical $q$-series identity from Ramanujan's lost notebook \cite[Entry 9.5.2]{An-Be},
	\[
		\sum_{n\ge0} \left(q;q^2\right)_nq^n = \sum_{n\ge0} (-1)^nq^{3n^2+2n}\left(1+q^{2n+1}\right).
	\]
	Using this along with \eqref{umnstargf1}, we have
	\begin{align*}
		\sum_{n\ge0} \ou^*(n)q^n &= \sum_{n\ge0} \left(-q;q^2\right)_n^2q^{2n+1} \equiv \sum_{n\ge0} \left(q^2;q^4\right)_nq^{2n+1} = \sum_{n\ge0} (-1)^nq^{6n^2+4n+1}\left(1+q^{4n+2}\right)\\
		&\equiv \sum_{n\in\Z} q^{6n^2+4n+1}\Pmod2.
	\end{align*}
	Now in the extremes of this string of equations we replace $q$ by $q^6$ and multiply by $q^{-2}$ to obtain
	\[
		\sum_{n\ge0} \ou^*(n)q^{6n-2} \equiv \sum_{n\in\Z} q^{(6n+2)^2}\Pmod2, 
	\]
	and the result follows.
\end{proof}

We now state the main result of this section.

\begin{theorem}\label{T:mod4congruences}
	Let $k\in\N$ and for $r$ with $1\le r \le k+1$, let $p_r\ge5$ be prime. For any $j\not\equiv0\Pmod{p_{k+1}}$, if $p_{k+1}\not\equiv7,13\Pmod{24}$ or $p_{k+1}\equiv7,13\Pmod{24}$ and $(\frac{3j}{p_{k+1}})=-1$, then we have:
	\begin{enumerate}
		\item If $j$ is odd, then we have
		\begin{equation*}
			{\rm ou}^*\left(4 p_1^2 \cdots p_{k+1}^2 n + 2 p_1^2 \cdots p_k^2 p_{k+1} j + \tfrac{8p_1^2 \cdots p_{k+1}^2 + 1}{3}\right) \equiv 0 \pmod{4}.
		\end{equation*}
		
		\item If $j$ is even, then we have
		\begin{equation*}
			{\rm ou}^*\left(4 p_1^2 \cdots p_{k+1}^2 n + 2 p_1^2 \cdots p_k^2 p_{k+1} j + \tfrac{2p_1^2 \cdots p_{k+1}^2 + 1}{3}\right) \equiv 0 \pmod{4}.
		\end{equation*}
	\end{enumerate}
\end{theorem}

To make the proof smoother, we first prove a simple lemma.

\begin{lemma}\label{L:oddpolynomial}
	For $n\in\N$ we have that modulo $4$,
	\begin{equation*}
		\left(-q;q^2\right)_n^2 - \left(-q^2;q^4\right)_n
	\end{equation*}
	is an odd polynomial.
\end{lemma}

\begin{proof}
	We prove the claim by induction. The case $n=1$ is clear. 

	Now assume that the claim holds for $n\in\N$. Then
	\[
		\left(-q;q^2\right)^2_{n+1}
		-\left(-q^2;q^4\right)_{n+1}
		=
		\left(1+q^{4n+2}\right)
		\left(
		\left(-q;q^2\right)^2_n
		-\left(-q^2;q^4\right)_n
		\right)
		+
		2q^{2n+1}\left(-q;q^2\right)^2_{n}.
	\]
	By the induction assumption the first term is an odd polynomial $\pmod{4}$. Thus we are left to show that $(-q;q^2)^2_n$ is an even polynomial $\pmod{2}$. Now we have the even polynomial 
	\[
		\left(-q;q^2\right)^2_n \equiv \left(-q^2;q^4\right)_n\Pmod2. \qedhere
	\]
\end{proof}

Next we use a result of Chen and Chen \cite{Ch-Ch}, who employed the theory of class numbers to prove congruences modulo $4$ for $\Ec\Oc(n)$, defined via the infinite product
\begin{equation} \label{EOproduct}
	\sum_{n\ge0} \Ec\Oc(n)q^n := \frac{\left(q^4;q^4\right)_\infty^3}{\left(q^2;q^2\right)_\infty^2}.
\end{equation}

\begin{theorem}[\cite{Ch-Ch}]\label{ChenChen}
	Let $k\in\N_0$. For $i$ with $1\le i \le k+1$, let $p_i \ge5$ be prime. For $j\not\equiv0\Pmod{p_{k+1}}$, if $p_{k+1}\not\equiv7,13\Pmod{24}$ or $p_{k+1}\equiv7,13\Pmod{24}$ and $(\frac{3j}{p_{k+1}})=-1$, then for $n\in\N_0$
	\begin{equation}\label{eq:EOmod4}
		\Ec\Oc\left(p_1^2\cdots p_{k+1}^2n+p_1^2\cdots p_k^2p_{k+1}j+\tfrac{p_1^2\cdots p_{k+1}^2-1}{3}\right) \equiv 0\Pmod4.
	\end{equation}
\end{theorem}

We are now ready to prove \Cref{T:mod4congruences}.

\begin{proof}[Proof of \Cref{T:mod4congruences}]
	First recall the third order mock theta function,
	\[
		\nu(q) := \sum_{n\ge0} \left(q;q^2\right)_n(-q)^n =: \sum_{n\ge0} c(n)q^n.
	\]
	Using Lemma \ref{L:oddpolynomial} we have that
	\[
		\sum_{n\ge0} \ou^*(n)q^n - \sum_{n\ge0}(-1)^nc(n)q^{2n+1} = \sum_{n\ge0} \left(-q;q^2\right)_n^2q^{2n+1} - \sum_{n\ge0} \left(-q^2;q^4\right)_nq^{2n+1}
	\]
	is supported on even exponents of $q$ modulo $4$. Therefore 
	\begin{equation} \label{ou^*(2n+1)}
		{\rm ou}^*(2n+1) \equiv (-1)^nc(n) \pmod{4}.
	\end{equation}

	Next we note that by \cite[equation (26.88)]{Fine} we have
	\[
		\sum_{n\ge0} c(2n)q^{2n} = \frac{\left(q^4;q^4\right)_\infty^3}{\left(q^2;q^2\right)_\infty^2}.
	\]
	This is the same product as in \eqref{EOproduct}, and so $c(n)=\Ec\Oc(n)$ if $n$ is even. Therefore the congruences for $\Ec\Oc(n)$ in \Cref{ChenChen} imply congruences for $c(n)$. The argument on the left-hand side of \eqref{eq:EOmod4} is even if and only if $n\equiv j\Pmod2$. So for $j$ even, we let $n\mapsto2n$ and for $j$ odd, we let $n\mapsto2n+1$ to obtain:
	\begin{enumerate}
		\item If $j$ is odd, then
		\[
			c\left(2p_1^2\cdots p_{k+1}^2n+p_1^2\cdots p_k^2p_{k+1}j+\tfrac{4p_1^2\cdots p_{k+1}^2-1}{3}\right) \equiv 0\Pmod4.
		\]
		
		\item If $j$ is even, then
		\[
			c\left(2p_1^2\cdots p_{k+1}^2n+p_1^2\cdots p_k^2p_{k+1}j+\tfrac{p_1^2\cdots p_{k+1}^2-1}{3}\right) \equiv 0\Pmod4.
		\]
	\end{enumerate}
	By \eqref{ou^*(2n+1)} the theorem follows.
\end{proof}

\section{Conclusion}

This paper contains a preliminary investigation of odd unimodal sequences, establishing generating functions, basic asymptotics, and some congruence properties modulo $4$. While other variants of unimodal sequences have arisen in the literature (e.g. \cite{BFW,BJ,KLL}), odd unimodal sequences are perhaps the most natural. Below we leave two ideas for further study. The motivated reader will surely find many more.

First, improve upon the asymptotics in Theorems \ref{T:theo3} and \ref{T:theo2} to find Rademacher-type formulas for ${\rm ou}(n)$ and ${\rm ou}^*(n)$. Up to simple terms the generating function for $\ou(n)$ and is a mixed false theta function while the generating function for $\ou^*(n)$ is a mixed mock theta function.   In both cases the weight is $\frac12$. One could now use the Circle Method to deduce asymptotic formulas for $\ou(n)$ and $\ou^*(n)$. However, an exact formula is out of reach with this method because the weight is too large. To obtain an exact formula one would need to find new methods (like Poincar\'e type series).

Second, it appears that the arithmetic progressions in some of the congruences in \Cref{T:mod4congruences} can be enlarged. For example, with $k=0$ (i.e., in the case of \Cref{T:mod4}) the congruences corresponding to the primes $5,7$, and $11$ are
\begin{align}
	{\rm ou}^*(100n+r) &\equiv 0 \pmod{4} \text{for $r \in \{37,57,77,97\}$}, \label{mod100} \\
	{\rm ou}^*(196n+r) &\equiv 0 \pmod{4} \text{for $r \in \{61,89,145\}$}, \label{mod196} \\
	{\rm ou}^*(484n+r) & \equiv 0 \pmod{4} \text{for $r \in \{125,169,213,257,301,345,389,433,477,521\}$}. \label{mod484}
\end{align}
Computations suggest that the cases $r \in \{37,97\}$ of \eqref{mod100} are special cases of the congruences
\begin{equation*}
	{\rm ou}^*(50n+r) \equiv 0 \pmod{4} \text{for $r \in \{37,47\}$},
\end{equation*}
the cases $r \in \{61,145\}$ of \eqref{mod196} are special cases of the congruences
\begin{equation*}
	{\rm ou}^*(98n + r) \equiv 0 \pmod{4} \text{for $r \in \{47,61\}$},
\end{equation*}
and all of the congruences in \eqref{mod484} are special cases of congruences in the corresponding arithmetic progressions modulo $242$. We leave it as an open problem to establish exactly which of the congruences in Theorem \ref{T:mod4congruences} (or Theorem \ref{T:mod4}) can be strengthened in this way.

%

\appendix
\section{}

Here we prove the second bound from \eqref{E:second2}. We require the relation $\b(x)=\erfc(\sqrt{\pi x})$ and use the bound $\erfc(x)\ll1$. To use \eqref{E:refine}, we write
\begin{align}\nonumber
	\Th^-\left(\tfrac{iz}{2\pi}\right) = &-2\sum_\pm \sum_{\d\in\{0,1\}} (-1)^\d\\
	& \left(\sum_{n_2\ge0} e^{-(n_2+1)^2z}\sum_{n_1\ge0} \b\left(\tfrac8\pi\left(n_1+\tfrac\d2+\tfrac14\right)^2x\right) e^{-12\left(n_1+\frac\d2+\frac14\right)^2z} e^{\mp8\left(n_1+\frac\d2+\frac14\right)(n_2+1)z}\right.\nonumber\\
	\label{eq:ow}
	&\qquad \left.\pm \sum_{n_2\ge0} \b\left(\tfrac{2}{3\pi}(n_2+1)^2x\right)e^{-(n_2+1)^2z} \sum_{n_1\ge0} e^{-12\left(n_1+\frac\d2+\frac14\right)^2z\mp8\left(n_1+\frac\d2+\frac14\right)(n_2+1)z}\right).
\end{align}
	
We write the first term as 
\[
	-2\sum_\pm \sum_{\d\in\{0,1\}} (-1)^\d\sum_{n_2\ge0} e^{-(n_2+1)^2z}\sum_{n_1\ge0} H_{n_2,\pm}\left(\left(n_1+\tfrac\d2+\tfrac14\right)\sqrt x\right),
\]
where
\[
	H_{n_2,\pm}(w_1) := \b\left(\tfrac{8w_1^2}{\pi}\right)e^{-12\frac zxw_1^2\mp8(n_2+1)\frac{z}{\sqrt x}w_1}.
\]
Now \eqref{E:refine} with $N=0$, gives that
\[
	\sum_{n_1\ge0} H_{n_2,\pm}\left(\left(n_1+\tfrac\d2+\tfrac14\right)\sqrt x\right) = \frac{1}{\sqrt x}\int_0^\infty H_{n_2,\pm}(w_1) dw_1 + \Ec\left(\tfrac\d2+\tfrac14;x\right),
\]
where
\begin{equation}\label{E:Ea}
	\Ec(a;x) := -\sum_{k_1\ge0} \frac{H_{n_2,\pm}^{(k_1)}(0)}{(k_1+1)!}a^{k_1+1}x^\frac{k_1}{2} - \frac1{\sqrt x}\int_{a\sqrt x}^\infty H_{n_2,\pm}(w_1)\wt B_0\left(\tfrac{w_1}{\sqrt x}-a\right) dw_1.
\end{equation}
Note that $\wt B_0(x)=1$. The contribution from the main term vanishes (because of the $(-1)^\d$).
	
The first term in the error contributes
\begin{align}\nonumber
	&2\sum_\pm \sum_{\d\in\{0,1\}} (-1)^\d\sum_{n_2\ge0} e^{-(n_2+1)^2z}\sum_{k_1\ge0} \left[\left(\tfrac{\del}{\del w_1}\right)^{k_1} H_{n_2,\pm}(w_1)\right]_{w_1=0} \frac{\left(\frac\d2+\frac14\right)}{(k_1+1)!}^{k_1+1}x^\frac{k_1}{2}\\
	\nonumber
	&= 2\sum_\pm \sum_{\d\in\{0,1\}} (-1)^\d\sum_{k_1\ge0} \frac{\left(\frac\d2+\frac14\right)}{(k_1+1)!}^{k_1+1}x^\frac{k_1}{2}\sum_{\substack{0\le\ell_1,\ell_2\le k_1\\\ell_1+\ell_2=k_1}} \binom{k_1}{\ell_1}\left[\left(\tfrac{\del}{\del w_1}\right)^{\ell_1} \left(\b\left(\tfrac{8w_1^2}{\pi}\right)e^{-12\frac {z}xw_1^2}\right)\right]_{w_1=0}\\
	\label{E:error1}
	&\hspace{8.8cm}\times \sum_{n_2\ge0} \left(\mp8(n_2+1)\frac{z}{\sqrt x}\right)^{\ell_2}e^{-(n_2+1)^2z}.
\end{align}
Now only $\ell_2$ even survive (otherwise the $\pm$ cancels). We now determine the asymptotic behaviors of ($\ell_2$ even)
\[
	\sum_{n_2\ge0}
	(n_2+1)^{\ell_2}
	e^{-(n_2+1)^2z}
	=
	(-1)^{\frac {\ell_2}2} 
	\left(\frac{\partial}{\partial z}\right)^{\frac{\ell_2}{2}}
	\sum_{n\ge1}
	e^{-n^2z}.
\]
For this recall the modular theta function
\[
	\vartheta(\tau) := \sum_{n\in\Z} e^{\pi i n^2 \tau}.
\]
It satisfies
\[
	\vartheta(\tau)=
	(-i\tau)^{-\tfrac12}\vartheta\left(-\tfrac1{\tau}\right).
\]
Thus
\begin{equation}\label{E:Thettran}
	\sum_{n\ge1} e^{-n^2z} = \tfrac12\sum_{n\in\Z} e^{-n^2z} - \tfrac12 = \tfrac12\vth\left(\tfrac{iz}{\pi}\right) - \tfrac12 = \tfrac12\left(\tfrac z\pi\right)^{-\tfrac12}\vth\left(\tfrac{\pi i}{z}\right) - \tfrac12.
\end{equation}
	
The second term contributes to \eqref{E:error1} (it only survives if $\ell_2=0$)
\begin{multline*}
	-\tfrac12\sum_\pm \sum_{\d\in\{0,1\}} (-1)^\d\sum_{k_1\ge0} \frac{\left(\frac \d2+\frac14\right)^{k_1+1}}{(k_1+1)!}x^\frac{k_1}{2}\left[\left(\tfrac{\del}{\del w_1}\right)^{k_1} \left(\b\left(\tfrac{8w_1^2}{\pi}\right)e^{-12\frac zxw_1^2}\right)\right]_{w_1=0}\\
	\ll \sum_{k_1\ge0} \frac{x^\frac{k_1}{2}}{(k_1+1)!}\left(1+\left|\frac zx\right|^{k_1}\right)\left[\left(\tfrac{\del}{\del w_1}\right)^{k_1} \left(\b\left(\tfrac{8w_1^2}{\pi}\right)e^{-12w_1^2}\right)\right]_{w_1=0} \ll 1.
\end{multline*}
	
The first term from \eqref{E:Thettran} contributes to \eqref{E:error1} (noting that $\ell$, $k_1$ need to be even)
\begin{equation}\label{first}
	\ll \sum_{k_1\geq0} \frac{x^{{k_1}}}{(2k_1+1)!}
	\sum_{0\leq\ell\leq k_1}
	\binom{2k_1}{2\ell}
	\left[\left(\tfrac{\partial}{\partial w_1}\right)^{2k_1-2\ell}\left(\b\left(\tfrac{8w_1^2}{\pi}\right)e^{-12\frac zx w_1^2}\right)\right]_{w_1=0} \left(\tfrac{8z}{\sqrt{x}}\right)^{2\ell} \frac{\partial^{2\ell}}{\partial z^{2\ell}} \frac{\vartheta \left(\frac{\pi i}{z}\right)}{z^{\frac12}}.
\end{equation}
Now assume that $\frac{x}{|z|^2}>1$ (which is true as $x\to0$). Then
\begin{equation*}
	z^{2\ell} \frac{\partial^{2\ell}}{\partial z^{2\ell}} \frac{\vartheta \left(\frac{\pi i}{z}\right)}{z^{\frac12}}
	\ll 
	\left[
	\frac{\partial^{2\ell}}{\partial w_2^{2\ell}} 
	\frac{\vartheta \left(\frac{\pi i}{w_2}\right)}{{w_2}^{\frac12}}
	\right]_{w_2=1}.
\end{equation*}
Moreover, as above,
\begin{equation*}
	\left[\left(\tfrac{\del}{\del w_1}\right)^{2k_1-2\ell} \left(\b\left(\tfrac{8w_1^2}{\pi}\right)e^{-12\frac zxw_1^2}\right)\right]_{w_1=0}\ll \left(1+\De^{2k_1-2\ell}\right)\left[\left(\tfrac{\del}{\del w_1}\right)^{2k_1-2\ell} \left(\b\left(\tfrac{8w_1^2}{\pi}\right)e^{-12 w_1^2}\right)\right]_{w_1=0}.
\end{equation*}
This yields that \eqref{first} can be bound against $O(1)$.
Thus the first term in \eqref{E:Ea} is overall $O(1)$.
	
The second term of \eqref{E:Ea} contributes 
\begin{align}\nonumber
	&\tfrac{2}{\sqrt x}\sum_\pm \sum_{\d\in\{0,1\}} (-1)^\d\sum_{n_2\ge0} e^{-(n_2+1)^2z}\int_{\left(\frac\d2+\frac14\right)\sqrt x}^\infty H_{n_2,\pm}(w_1) dw_1\\
	\nonumber
	&\hspace{3cm}= \tfrac{2}{\sqrt x}\sum_\pm \sum_{n_2\ge0} e^{-(n_2+1)^2z}\left(\int_\frac{\sqrt x}{4}^\infty-\int_\frac{3\sqrt x}{4}^\infty\right)H_{n_2,\pm}(w_1)dw_1\\
	\label{E:second}
	&\hspace{3cm}= \tfrac{2}{\sqrt x}\sum_\pm \int_\frac{\sqrt x}{4}^\frac{3\sqrt x}{4} \b\left(\tfrac8{\pi}w_1^2\right) e^{-12\frac zxw_1^2}\sum_{n_2\ge0} e^{-(n_2+1)^2z\mp8(n_2+1)\frac{z}{\sqrt x}w_1} dw_1.
\end{align}
	
The sum on $n_2$ may be written as
\[
	\sum_{n_2\ge0} h_{[1]}\left((n_2+1)\sqrt x\right),
\]
where $h_{[1]}(w_2):=e^{-\frac zxw_2^2\mp8\frac zxw_1w_2}$. Using \eqref{E:refine}, with $N=0$, we have
\[
	\sum_{n_2\ge0} h_{[1]}\left((n_2+1)\sqrt x\right) = \frac{1}{\sqrt x}\int_0^\infty h_{[1]}(w_2) dw_2 + \Ec^{[1]}(x),
\]
where
\begin{equation*}
	\Ec^{[1]}(x):= -\sum_{k_2\ge0}
	\frac{h_{[1]}^{(k_2)}(0)}{(k_2+1)!}
	x^{\frac {k_2}2}
	-\frac1{\sqrt{x}}
	\int_{\sqrt{x}}^\infty
	h_{[1]}(w_2)dw_2.
\end{equation*}
	
The main term contributes to \eqref{E:second} as
\begin{equation*}
	\tfrac2x\sum_\pm \int_\frac{\sqrt x}{4}^\frac{3\sqrt x}{4}\b\left(\tfrac{8w_1^2}{\pi}\right)e^{-12\frac zxw_1^2}\int_0^\infty h_{[1]}(w_2) dw_2dw_1 = \tfrac2x\sum_\pm \int_\frac{\sqrt x}{4}^\frac{3\sqrt x}{4}\b\left(\tfrac{8w_1^2}{\pi}\right)e^{4\frac zxw_1^2}\int_0^\infty e^{-\frac zx(w_2\pm 4w_1)^2} dw_2.
\end{equation*}
Using that $\int_{4w_1}^\infty+\int_{-4w_1}^\infty=2\int_0^\infty$, the above is in $O(\frac{1}{\sqrt z})$. 
	
We next consider the first term in $\Ec^{[1]}(x)$ which contributes
\[
	-\tfrac{2}{\sqrt x}\sum_\pm \int_\frac{\sqrt x}{4}^\frac{3\sqrt x}{4}\b\left(\tfrac{8w_1^2}{\pi}\right)e^{-12\frac zxw_1^2}\sum_{k_2\ge0} \left[\tfrac{\del^{k^2}}{\del w_2^{k^2}} e^{-\frac zxw_2^2\mp8\frac zxw_1w_2}\right]_{w_2=0} dw_1\frac{x^\frac{k_2}{2}}{(k_2+1)!}.
\]
We have 
\[
	\left[\pd{w_2}{k_2} e^{-\frac zxw_2^2\mp8\frac zxw_1w_2}\right]_{w_2=0} = \sum_{\ell=0}^k \binom{k_2}{\ell}\left(\mp8\frac zxw_1\right)^\ell\leg{\sqrt z}{\sqrt x}^{k_2-\ell}\left[\pd{w_2}{k_2-\ell} e^{-w_2^2}\right]_{w_2=0}.
\]
Now the $\pm$ enforces $\ell$, $k_2-\ell$ to be even because of the $\pm$ and bound overall against $O(\frac{1}{\sqrt z})$.
	
We next consider the second term in $\Ec^{[1]}(x)$. This contributes 
\begin{equation}\label{E:contributes}
	\tfrac1x \int_\frac{\sqrt x}{4}^\frac{3\sqrt x}{4} \b\left(\tfrac{8w_1^2}{\pi}\right)e^{-12\frac zxw_1^2}\int_{\sqrt x}^\infty e^{-\frac zxw_2^2\mp8\frac zxw_1w_2} dw_2 dw_1.
\end{equation}
We write 
\[
	e^{-12\frac zxw_1^2}\int_{\sqrt x}^\infty e^{-\frac zxw_2^2\mp8\frac zxw_1w_2} dw_2 \ll e^{4\frac zxw_1^2}\int_{\sqrt x}^\infty e^{-\frac zx(w_2\pm4w_1)^2} dw_2 \ll e^{4w_1^2}\int_\R e^{-w_2^2} dw_2.
\]
Thus \eqref{E:contributes} may be bound against $O(\frac{1}{\sqrt z})$. Combining gives that the first term is $O(\frac{1}{\sqrt z})$.
	
We next turn to the second term in \eqref{eq:ow} and proceed as before, again first considering the sum on $n_1$. We have
\[
	\sum_{n_1\ge0} e^{-12\left(n_1+\frac\d2+\frac14\right)^2z\mp8\left(n_1+\frac\d2+\frac14\right)(n_2+1)z} = \sum_{n_1\ge0} G_{n_2,\pm}\left(\left(n_1+\tfrac\d2+\tfrac14\right)\sqrt z\right),
\]
where $G_{n_2,\pm}(w_1):=e^{-12w_1^2\mp8(n_2+1)\sqrt zw_1}$. Now \eqref{E:refine} gives (with $N=0$) that
\[
	\sum_{n_1\ge0} G_{n_2,\pm}\left(\left(n_1+\tfrac\d2+\tfrac14\right)\sqrt z\right) = \frac{1}{\sqrt z}\int_0^\infty G_{n_2,\pm}(w_1) dw_1 + E\left(\tfrac \d2+\tfrac14;z\right),
\]
where 
\[
	E(a;z) := -\sum_{k_1\ge0} \frac{G_{n_2,\pm}^{(k_1)}(0)}{(k_1+1)!}a^{k_1+1}z^\frac{k_1}{2} - \frac{1}{\sqrt z}\int_{a\sqrt{z}}^{\sqrt z\infty} G_{n_2,\pm}(w_1)\wt B_0\left(\tfrac{w_1}{\sqrt z}-a\right) dw_1.
\]
	
The main term contributes overall
\[
	-\tfrac{2}{\sqrt z}\sum_\pm \pm \sum_{\d\in\{0,1\}} (-1)^\d\sum_{n_2\ge0} \b\left(\tfrac{2}{3\pi}(n_2+1)^2x\right)e^{-(n_2+1)^2z}\int_0^\infty G_{n_2,\pm}(w_1) dw_1 = 0.
\]
	
The first term in the error $E(\frac \d2 +\frac14;z)$ contributes
\begin{multline*}
	2\sum_\pm \pm\sum_{\d \in \{0,1\}}(-1)^\d \sum_{n_2\ge0} \b\left(\tfrac{2}{3\pi}(n_2+1)^2x\right)e^{-(n_2+1)^2z}\sum_{k_1\ge0} \left[\left(\tfrac{\del}{\del w_1}\right)^{k_1} G_{n_2,\pm}(w_1)\right]_{w_1=0}\\
	\times \frac{\left(\frac\d2+\frac14\right)^{k_1+1}}{(k_1+1)!}z^\frac{k_1}{2} = 2\sum_\pm \pm \sum_{\d\in\{0,1\}} (-1)^\d\sum_{k_1\ge0} \left(\tfrac\d2+\tfrac14\right)^{k_1+1}\frac{z^\frac{k_1}{2}}{(k_1+1)!}\\
	\times \left[\left(\tfrac{\del}{\del w_1}\right)^{k_1} \left(e^{-12w_1^2}\sum_{n_2\ge0} \b\left(\tfrac{2}{3\pi}(n_2+1)^2x\right)e^{-(n_2+1)^2z\mp8(n_2+1)\sqrt zw_1}\right)\right]_{w_1=0}.
\end{multline*}
We write the sum on $n_2$ as
\[
	\sum_{n_2\ge0} g_{[1],\pm}\left((n_2+1)\sqrt x\right),
\]
where $g_{[1],\pm}(w_2) := \b(\tfrac{2w_2^2}{3\pi})e^{-\frac {z}xw_2^2\mp8\frac{\sqrt z}{\sqrt x}w_1w_2}$.
Now \eqref{E:refine}, with $N=0$, gives that
\[
	\sum_{n_2\ge0} g_{[1],\pm}\left((n_2+1)\sqrt x\right) = \frac{1}{\sqrt x}\int_0^\infty g_{[1],\pm}(w_2) dw_2 + E^{[1]}(x),
\]
where
\[
	E^{[1]}(x) := -\sum_{k_2\ge0} \frac{g_{[1],\pm}^{(k_2)}(0)}{(k_2+1)!}x^\frac{k_2}{2} - \frac1{\sqrt{x}}\int_{\sqrt x}^\infty g_{[1],\pm}(w_2)\wt B_0\left(\tfrac{w_2}{\sqrt x}-1\right) dw_2.
\]
	
The main term contributes
\begin{equation}\label{E:main}
	\tfrac{2}{\sqrt x}\sum_\pm \pm \sum_{\d\in\{0,1\}} (-1)^\d\sum_{k_1\ge0} \left(\tfrac\d2+\tfrac14\right)^{k_1+1}\frac{z^\frac{k_1}{2}}{(k_1+1)!} \left[\left(\tfrac{\del}{\del w_1}\right)^{k_1} \left(e^{-12w_1^2}\int_0^\infty g_{[1],\pm}(w_2) dw_2\right)\right]_{w_1=0}\hspace{-0.1cm}.\hspace{-0.2cm}
\end{equation}
We rewrite
\begin{multline*}
	\left[\left(\tfrac{\del}{\del w_1}\right)^{k_1} \left(e^{-12w_1^2}\int_0^\infty g_{[1],\pm}(w_2) dw_2\right)\right]_{w_1=0}\\
	=\sum_{j=0}^{k_1} \pmat{k_1\\j}\left[\left(\tfrac{\del}{\del w_1}\right)^{k_1-j} e^{-12w_1^2}\right]_{w_1=0}\int_0^\infty \b\left(\tfrac{2w_2^2}{3\pi}\right)e^{-\frac zxw_2^2}\left(\mp8\sqrt{\tfrac zx}w_2\right)^j dt_2.
\end{multline*}
The $\pm$ forces $j$ to be even. Also $k_1$ is even. Thus \eqref{E:main} is $O(\frac{1}{\sqrt z})$.

The first term from $E^{[1]}$ contributes
\begin{multline*}
	2\sum_\pm \pm \sum_{\d\in\{0,1\}} (-1)^\d\sum_{k_1\ge0} \left(\tfrac\d2+\tfrac14\right)^{k_1+1} \Bigg[\left(\tfrac{\del}{\del w_1}\right)^{k_1}\\
	\Bigg(e^{-12w_1^2} \sum_{k_2\ge0} \left[\left(\tfrac\del{\del w_2}\right)^{k_2} \left(g_{[1],\pm}(w_2)\right)\right]_{w_2=0}\Bigg)\Bigg]_{w_1=0}\frac{z^\frac{k_1}{2}}{(k_1+1)!}\frac{x^\frac{k_2}{2}}{(k_2+1)!} \ll 1.
\end{multline*}
	
The second term from $E^{[1]}$ contributes
\begin{equation*}
	\tfrac{2}{\sqrt x}\sum_\pm \pm \sum_{\d\in\{0,1\}}(-1)^\d \sum_{k_1\ge0} \left(\tfrac\d2+\tfrac14\right)^{k_1+1}\frac{z^\frac{k_1}{2}}{(k_1+1)!}	{\vphantom{\left(\frac{\del}{\del w_2}\right)^{k_2}}} \left[\left(\tfrac{\del}{\del w_1}\right)^{k_1}\left(e^{-12w_1^2}\int_{\sqrt x}^\infty g_{[1],\pm}(w_2) dw_2\right)\right]_{w_1=0} .
\end{equation*}
	
We compute 
\begin{equation*}
	\left[
	\left(\tfrac{\del}{\del w_1}\right)^{k_1}
	e^{-12w_1^2\mp 8\frac{\sqrt{z}}{\sqrt{x}}w_1w_2}
	\right]_{w_1=0}
	=
	\sum_{\ell=0}^{k_1}
	\binom{k_1}{\ell}
	\left[
	\left(\tfrac{\del}{\del w_1}\right)^{k_1-\ell}
	e^{-12w_1^2}
	\right]_{w_1=0}
	\left(\mp 8\tfrac{\sqrt{z}}{\sqrt{x}}w_2\right)^\ell.
\end{equation*}
Arguing as above we need $\ell$, $k_1$ to be even and obtain overall $O(\frac{1}{\sqrt z})$.

The second term in the error $E(\frac\d2+\frac14;z)$ contributes
\begin{multline*}
	\tfrac{2}{\sqrt z}\sum_\pm \pm \sum_{\d\in\{0,1\}} (-1)^\d\sum_{n_2\ge0} \b\left(\tfrac{2}{3\pi}(n_2+1)^2x\right)e^{-(n_2+1)^2z}\int_{\left(\frac\d2+\frac14\right)\sqrt z}^{\sqrt z\infty} e^{-12w_1^2\mp8(n_2+1)\sqrt zw_1} dw_1\\
	= \tfrac{2}{\sqrt z}\sum_\pm \pm \sum_{n_2\ge0} \b\left(\tfrac{2}{3\pi}(n_2+1)^2x\right)e^{-(n_2+1)^2z}\int_\frac{\sqrt z}{4}^\frac{3\sqrt z}{4} e^{-12w_1^2\mp8(n_2+1)\sqrt zw_1} dw_1.
\end{multline*}
	
Now the sum on $n_2$ is 
\[
	\sum_{n_2\ge0} g_{[2]}\left((n_2+1)\sqrt x\right),
\]
where
\[
	g_{[2]}(w_2) := \b\left(\tfrac{2w_2^2}{3\pi}\right)e^{-\frac zxw_2^2\mp8\frac{\sqrt z}{\sqrt x}w_1w_2}.
\]
Thus \eqref{E:refine}, with $N=0$, gives
\[
	\sum_{n_2\ge0} g_{[2]}\left((n_2+1)\sqrt x\right) = \frac{1}{\sqrt x}\int_0^\infty g_{[2]}(t_2) dt_2 + E^{[4]}(x),
\]
where
\[
	E^{[4]}(x) := -\sum_{k_2\ge0} \frac{g_{[2]}^{(k_2)}(0)}{(k_2+1)!}x^\frac{k_2}{2} - \frac{1}{\sqrt x}\int_{\sqrt x}^{\infty} g_{[2]}(w_2)\wt B_0\left(\tfrac{w_2}{\sqrt x}-1\right) dw_2.
\]
	
The main term contributes
\[
	\tfrac{2}{\sqrt z\sqrt x}\sum_\pm \pm \int_\frac{\sqrt z}{4}^\frac{3\sqrt z}{4} e^{-12w_1^2}\int_0^\infty g_{[2]}(w_2) dw_2 dw_1.
\]
We write 
\[
	e^{-12w_1^2}\int_0^\infty g_{[2]}(w_2) dw_2 = e^{4w_1^2}\int_0^\infty \b\left(\tfrac{2w_2^2}{3\pi}\right)e^{-\frac zx\left(w_2\mp4\frac{\sqrt x}{\sqrt z}\right)^2} dw_2 \ll e^{4w_1^2}.
\]
Thus we have overall $O(\frac{1}{\sqrt z})$.
	
The first term in the error $E^{[4]}$ contributes
\[
	-\tfrac{2}{\sqrt z}\sum_\pm \pm \int_\frac{\sqrt z}{4}^\frac{3\sqrt z}{4} e^{-12w_1^2}\sum_{k_2\ge0} \left[\left(\tfrac\del{\del w_2}\right)^{k_2} \b\left(\tfrac{2w_2^2}{3\pi}\right) e^{-\frac zxw_2^2\mp8\frac{\sqrt z}{\sqrt x}w_1w_2}\right]_{w_2=0}\frac{x^\frac{k_2}{2}}{(k_2+1)!}.
\]
We bound, since $|z|\leq (1+\Delta)x$,
\begin{equation*}
	\left[\left(\tfrac\del{\del w_2}\right)^{k_2} \left(\b\left(\tfrac{2w_2^2}{3\pi}\right)e^{-\frac zxw_2^2\mp8\frac{\sqrt z}{\sqrt x}w_1w_2}\right)\right]_{w_2=0}	\ll \left(1+\De^\frac{k_2}{2}\right)\left[\left(\tfrac\del{\del w_2}\right)^{k_2} \b\left(\tfrac{2w_2^2}{3\pi}\right)e^{-w_2^2\mp8w_1w_2}\right]_{w_2=0}.
\end{equation*}
Then overall we have 
\[
	\ll \tfrac{1}{\sqrt z}\sum_{k_2\ge0} \left(1+\De^\frac{k_2}{2}\right)\left[\left(\tfrac\del{\del w_2}\right)^{k_2} \left(\b\left(\tfrac{2w_2^2}{3\pi}\right)e^{-w_2^2}\int_\frac{\sqrt z}{4}^\frac{3\sqrt z}{4} e^{-12w_1^2\mp8w_1w_2} dw_1\right)\right]_{w_2=0}\frac{x^\frac{k_2}{2}}{(k_2+1)!} \ll 1.
\]

The second term in the error $E^{[4]}$ can be estimated as before as $O(\frac{1}{\sqrt z})$.

\end{document}